\documentclass[12pt,a4paper,reqno]{article}

\usepackage{graphicx} 

\usepackage{mathtools}
\usepackage{amsmath}
\usepackage{amssymb}
\usepackage{amsthm}
\usepackage{tikz,pgfplots}
\usetikzlibrary{decorations.markings}
\usepackage{comment}
\usepackage{url}
\usepackage{physics}
\usepackage{subfig}
\usepackage{float}

\usepackage{thmtools}
\usepackage{thm-restate}

\usepackage{graphicx,tikz}
\usetikzlibrary{shadows}
\usepackage{tkz-graph}

\usepackage{caption}
\usepackage{todonotes}
\usetikzlibrary{arrows}
\usepackage[utf8]{inputenc}
\usepackage[english]{babel}
\usepackage{array}
\usepackage{booktabs}
\usepackage{mdframed}
\usepackage{enumitem}

\usepackage[margin=3cm]{geometry}

\usepackage{caption}
\usepackage{soul}
\tikzset{Bullet/.style={fill=black,draw,color=#1,circle,minimum size=20pt,scale=.5}}

\usepackage{hyperref}

\usepackage{csquotes}
\usepackage[sorting=nyt, giveninits=true, eprint=false, url=false, doi=false,isbn=false, maxbibnames=9]{biblatex}
\AtEveryBibitem{\clearfield{month}
\clearfield{day}
\clearlist{language}}

\hypersetup{
  colorlinks   = true, 
  urlcolor     = blue, 
  linkcolor    = black, 
  citecolor   = red 
}
\setstcolor{red}

\usepackage{seqsplit}
\usepackage{xparse}

\NewDocumentEnvironment{graphsix}{O{1em}+b} 
  {%
    \par\addvspace{8pt}            
    \begingroup
      \ttfamily\small
      \setlength{\leftskip}{#1}    
      \setlength{\parindent}{0pt}  
      \rightskip=0pt plus 1fil
      \seqsplit{#2}\par           
  }{%
    \par\addvspace{8pt}            
    \endgroup
  }

\newtheorem{theorem}{Theorem}
\newtheorem{definition}[theorem]{Definition}

\newtheorem{lemma}[theorem]{Lemma}
\newtheorem{conjecture}[theorem]{Conjecture}

\newtheorem{remark}[theorem]{Remark}

\usepackage{booktabs,array,seqsplit}

\newcolumntype{T}[1]{>{\raggedright\arraybackslash}p{#1}}

\DeclareMathOperator{\p}{P}
\DeclareMathOperator{\cyp}{CP}
\DeclareMathOperator{\gp}{GP}
\DeclareMathOperator{\cgp}{CGP}
\DeclareMathOperator{\sgp}{SGP}

\newcommand{\ZZ}{\mathbb{Z}}
\newcommand{\mL}{\mathcal{L}}
\newcommand{\bG}{\mathbf{G}}

\usepackage{authblk}

\title{\bfseries\Large Infinitely many counterexamples\\[2pt]
        to a conjecture of Lov\'asz}

\author[1,2]{Aida~Abiad}
\author[3]{Frederik~Garbe}
\author[4]{Xavier~Povill}
\author[5]{Christoph~Spiegel}

\affil[ ]{%
  \vspace{0.8em}\mbox{}    
  \scriptsize {\tt a.abiad.monge@tue.nl} \,
          {\tt garbe@informatik.uni-heidelberg.de} \,
          {\tt xavier.povill@upc.edu} \,
          {\tt spiegel@zib.de}}
          
\affil[1]{Department of Mathematics and Computer Science, Eindhoven University of Technology}
\affil[2]{Department of Mathematics and Data Science, Vrije Universiteit Brussel}
\affil[3]{Institute of Computer Science, Heidelberg University}
\affil[4]{Department of Mathematics, Universitat Politècnica de Catalunya}
\affil[5]{Department AI in Society, Science, and Technology, Zuse Institute Berlin}

\date{}

\begin{document}

\maketitle

\begin{abstract}
Motivated by Ryser’s conjecture, which bounds the minimum size of a vertex cover in terms of the matching number in $r$-partite hypergraphs, Lov\'asz conjectured the stronger statement in 1975 that one can always reduce the matching number by removing $r-1$ vertices.
Clow, Haxell, and Mohar very recently disproved this for $r = 3$ using the explicit counterexample of a line hypergraph of a $3$-regular graph of order $102$.
We construct the first infinite family of counterexamples for $r = 3$,
the smallest of which is the line hypergraph of a graph of order only $22$.
In addition, we give the first counterexamples for $r = 4$ derived by computationally exploring a special class of permutation graphs.

\end{abstract}

\section{Introduction}

Let $H$ be an $r$-uniform hypergraph.
The \emph{matching number} $\nu(H)$ is the maximum number of pairwise disjoint edges and the \emph{vertex cover number} $\tau(H)$ is the minimum cardinality of a set of vertices intersecting every edge of $H$.
It follows immediately, that
\begin{equation*}
    \nu(H) \leq \tau(H) \leq r \, \nu (H),
\end{equation*}
and both inequalities are easily shown to be tight in general.
However, for the particular case of \emph{$r$-partite} $H$, i.e., if there exists a partition  $V_H = V_1\sqcup\ldots\sqcup V_r$ of the vertex set such that $|e\cap V_i|=1$ for every $1 \le i \le r$ and $e\in E_H$, Ryser conjectured, first explicitly stated in Henderson's 1971 dissertation~\cite{Hen71}, that
\begin{equation*}
    \tau(H)\leq (r-1) \, \nu(H).
\end{equation*}
For $r = 2$ this is true due to K\H onig's theorem and for $r=3$ it was proven by Aharoni~\cite{Aha2001}. Beyond this, the conjecture is only known to hold for specific families of hypergraphs~\cite{FraHerMcKWan17,BisDasMorSza21}. 

A potential proof strategy for Ryser's conjecture could consist of constructing the vertex cover by iteratively removing $r-1$ vertices in such a way that it decreases the matching number. Lov\'asz,  in his 1975 dissertation~\cite{Lov75} and later popularized through a 1988 survey of F\"uredi~\cite{furedi1988matchings}, conjectured that this should in fact be possible.

\begin{conjecture}(Lov\'asz 1975)\label{conj:Lovasz}
    For any $r$-partite hypergraph containing at least one edge, there exists a set of $r-1$ vertices whose deletion reduces its matching number.
\end{conjecture}

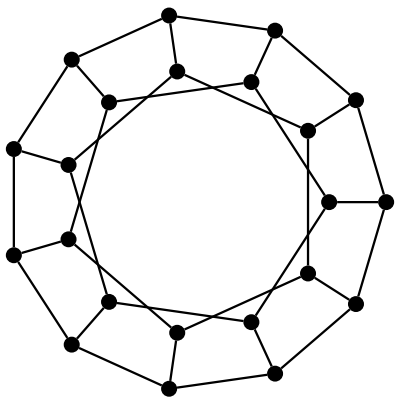
\begin{figure}[htp]
\centering
\begin{tikzpicture}[scale=0.5]
  \def\N{11}        
  \def\step{2}      
  \def\Rout{5}      
  \def\Rin{3.5}     

  \tikzset{
    VertexStyle/.style = {shape=circle, fill=black,
                          minimum size = 6pt, inner sep = 0pt}
  }

  \foreach \i in {0,...,\numexpr\N-1} {%
    \pgfmathsetmacro\theta{360/\N*\i}

    \pgfmathsetmacro\xO{\Rout*cos(\theta)}
    \pgfmathsetmacro\yO{\Rout*sin(\theta)}
    \Vertex[NoLabel,x=\xO cm,y=\yO cm]{u\i}

    \pgfmathsetmacro\xI{\Rin*cos(\theta)}
    \pgfmathsetmacro\yI{\Rin*sin(\theta)}
    \Vertex[NoLabel,x=\xI cm,y=\yI cm]{v\i}
  }

  \foreach \i in {0,...,\numexpr\N-1} {%
    \pgfmathtruncatemacro\next{mod(\i+1,\N)}   
    \pgfmathtruncatemacro\jump{mod(\i+\step,\N)}

    \Edge[lw=0.03cm](u\i)(u\next)

    \Edge[lw=0.03cm](u\i)(v\i)

    \Edge[lw=0.03cm](v\i)(v\jump)
  }
\end{tikzpicture}
\caption{The $3$-regular generalized Petersen graph $\gp (11,2)$ on $22$ vertices whose line hypergraph is a counterexample for Lov\'asz' conjecture for $r = 3$. It generalizes to the infinite family of graphs $\gp (5k + 11,2)$ whose line hypergraphs are counterexamples to Lov\'asz' conjecture for $r = 3$ and any $k \ge 0$.}
\label{fig:counterxample22vertices}
\end{figure}

The case $r=2$ also follows from K\H onig's theorem, since removing any vertex from a minimum vertex cover creates a graph with vertex cover number, and therefore also matching number, reduced by one. However, despite the result of Aharoni~\cite{Aha2001}, this conjecture remained open even for $r = 3$ until very recently.  As it turns out, with good reason: Clow, Haxell, and Mohar~\cite{CHM2025} found two counterexamples for the case $r=3$, namely the line hypergraphs of the Biggs-Smith graph and {\tt F168D}\footnote{A {\tt graph6} string and some properties of the {\tt F168D} graph can be found at \url{http://atlas.gregas.eu/graphs/65}.}, two cubic graphs of respective order $102$ and $168$. Several natural questions arise:

\begin{enumerate}\itemsep0pt
    \item Can one can construct an infinite family of counterexamples?~\cite[Question 4.1]{CHM2025}
    \item Are there counterexamples for any $r \ge 4$?~\cite[Question 4.2]{CHM2025}
    \item Are the properties of these two graphs, in particular that they are not Cayley, relevant to this application?~\cite[Question 4.4]{CHM2025}.
\end{enumerate}

In this paper we answer all the above questions. In particular, we give a positive answer to the first question by showing that the line hypergraphs of the generalized Petersen graphs $\gp (5k + 11,2)$ form such an infinite family for any $k \ge 0$. Note that the smallest counterexample in this family is of order $33$ coming from a cubic graph of order $22$, see Figure~\ref{fig:counterxample22vertices}, making it significantly smaller and much easier to analyze than the line hypergraph of the Biggs-Smith graph.

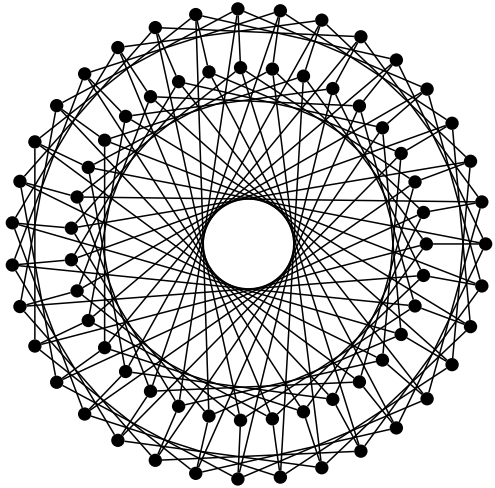
\begin{figure}[htp]
\centering
\begin{tikzpicture}[scale=0.52]
  \def\N{35}       
  \def\kOut{5}     
  \def\kIn{7}      
  \def\shift{15}   
  \def\Rout{6.0}   
  \def\Rin{4.5}    
  \tikzset{
    VertexStyle/.style = {shape=circle, fill=black,
                          minimum size=5pt, inner sep=0pt}
  }
  \foreach \i in {0,...,\numexpr\N-1} {
    \pgfmathsetmacro\ang{360/\N*\i}

    \pgfmathsetmacro\xO{\Rout*cos(\ang)}
    \pgfmathsetmacro\yO{\Rout*sin(\ang)}
    \Vertex[NoLabel,x=\xO cm,y=\yO cm]{u\i}

    \pgfmathsetmacro\xI{\Rin*cos(\ang)}
    \pgfmathsetmacro\yI{\Rin*sin(\ang)}
    \Vertex[NoLabel,x=\xI cm,y=\yI cm]{v\i}
  }

  \foreach \i in {0,...,\numexpr\N-1} {
    \pgfmathtruncatemacro\oNext{mod(\i+\kOut,\N)}
    \pgfmathtruncatemacro\iNext{mod(\i+\kIn,\N)}
    \pgfmathtruncatemacro\cross{mod(\i+\shift,\N)}

    \Edge[lw=0.02cm](u\i)(u\oNext)
    \Edge[lw=0.02cm](v\i)(v\iNext)
    \Edge[lw=0.02cm](u\i)(v\i)
    \Edge[lw=0.02cm](u\i)(v\cross)
  }
\end{tikzpicture}
\caption{The $4$-regular graph
\(\cgp (2,\ZZ_{35};5,7;\,15,0)\)
on $70$ vertices whose line hypergraph is a counterexample to Lov\'asz' conjecture for $r = 4$.}
\label{fig:counterxample70vertices}
\end{figure}

\begin{theorem}\label{thm:main}
    There exists an infinite family of cubic graphs whose line hypergraphs are counterexamples to Lov\'asz' conjecture for $r = 3$. 
\end{theorem}

We also give a first answer to the second question by providing several specific counterexamples for $r=4$, likewise derived as the line hypergraphs of regular graphs.
These constructions are significantly more complex to describe and the smallest of them has order $140$ coming from a quartic graph of order $70$, see \autoref{fig:counterxample70vertices}.

\begin{theorem}\label{thm:r=4}
    There exist quartic graphs whose line hypergraphs are counterexamples to Lov\'asz' conjecture for $r = 4$. 
\end{theorem}

We can also give an answer to the third question: while the generalized Petersen graphs we use for \autoref{thm:main} are not Cayley, we did find a specific example of a cubic Cayley graph on $36$ vertices, see Figure~\ref{fig:counterxample36vertices}, as well as a quartic Cayley graph on $96$ vertices.
Furthermore, while the Biggs-Smith graph and {\tt F168}D are both edge-transitive (and the Biggs-Smith graph also distance-transitive),
none of our examples are. The Biggs-Smith graph and {\tt F168}D are also Hamiltonian while the members $\gp (5 k + 11, 2)$ of our infinite family are not whenever $k \equiv 0 \mod 6$.

\begin{figure}[htp]
\centering
\definecolor{light_red}{RGB}{254,199,184}
\begin{tikzpicture}[scale=0.5]
\tikzset{VertexStyle/.style = {shape = circle,fill = black,minimum size = 6pt, inner sep = 0pt}}
\Vertex[NoLabel,x=2.5cm,y=0.6699cm]{v0}
\Vertex[NoLabel,x=1.7861cm,y=1.1698cm]{v1}
\Vertex[NoLabel,x=3.2899cm,y=0.3015cm]{v2}
\Vertex[NoLabel,x=7.5cm,y=9.3301cm]{v3}
\Vertex[NoLabel,x=9.924cm,y=5.8682cm]{v4}
\Vertex[NoLabel,x=9.6985cm,y=6.7101cm]{v5}
\Vertex[NoLabel,x=1.1698cm,y=8.2139cm]{v6}
\Vertex[NoLabel,x=8.2139cm,y=1.1698cm]{v7}
\Vertex[NoLabel,x=1.7861cm,y=8.8302cm]{v8}
\Vertex[NoLabel,x=8.8302cm,y=1.7861cm]{v9}
\Vertex[NoLabel,x=0.076cm,y=4.1318cm]{v10}
\Vertex[NoLabel,x=0.3015cm,y=3.2899cm]{v11}
\Vertex[NoLabel,x=2.5cm,y=9.3301cm]{v12}
\Vertex[NoLabel,x=9.3301cm,y=2.5cm]{v13}
\Vertex[NoLabel,x=0.6699cm,y=7.5cm]{v14}
\Vertex[NoLabel,x=7.5cm,y=0.6699cm]{v15}
\Vertex[NoLabel,x=10.0cm,y=5.0cm]{v16}
\Vertex[NoLabel,x=9.3301cm,y=7.5cm]{v17}
\Vertex[NoLabel,x=0.6699cm,y=2.5cm]{v18}
\Vertex[NoLabel,x=0.0cm,y=5.0cm]{v19}
\Vertex[NoLabel,x=8.8302cm,y=8.2139cm]{v20}
\Vertex[NoLabel,x=9.924cm,y=4.1318cm]{v21}
\Vertex[NoLabel,x=9.6985cm,y=3.2899cm]{v22}
\Vertex[NoLabel,x=3.2899cm,y=9.6985cm]{v23}
\Vertex[NoLabel,x=6.7101cm,y=0.3015cm]{v24}
\Vertex[NoLabel,x=0.3015cm,y=6.7101cm]{v25}
\Vertex[NoLabel,x=5.0cm,y=10.0cm]{v26}
\Vertex[NoLabel,x=5.0cm,y=0.0cm]{v27}
\Vertex[NoLabel,x=5.8682cm,y=9.924cm]{v28}
\Vertex[NoLabel,x=5.8682cm,y=0.076cm]{v29}
\Vertex[NoLabel,x=4.1318cm,y=0.076cm]{v30}
\Vertex[NoLabel,x=4.1318cm,y=9.924cm]{v31}
\Vertex[NoLabel,x=6.7101cm,y=9.6985cm]{v32}
\Vertex[NoLabel,x=8.2139cm,y=8.8302cm]{v33}
\Vertex[NoLabel,x=0.076cm,y=5.8682cm]{v34}
\Vertex[NoLabel,x=1.1698cm,y=1.7861cm]{v35}
\Edge[lw=0.03cm](v0)(v1)
\Edge[lw=0.03cm](v0)(v2)
\Edge[lw=0.03cm](v0)(v3)
\Edge[lw=0.03cm](v1)(v35)
\Edge[lw=0.03cm](v1)(v31)
\Edge[lw=0.03cm](v2)(v34)
\Edge[lw=0.03cm](v2)(v30)
\Edge[lw=0.03cm](v3)(v32)
\Edge[lw=0.03cm](v3)(v33)
\Edge[lw=0.03cm](v4)(v16)
\Edge[lw=0.03cm](v4)(v5)
\Edge[lw=0.03cm](v4)(v6)
\Edge[lw=0.03cm](v5)(v17)
\Edge[lw=0.03cm](v5)(v7)
\Edge[lw=0.03cm](v6)(v8)
\Edge[lw=0.03cm](v6)(v14)
\Edge[lw=0.03cm](v7)(v9)
\Edge[lw=0.03cm](v7)(v15)
\Edge[lw=0.03cm](v8)(v11)
\Edge[lw=0.03cm](v8)(v12)
\Edge[lw=0.03cm](v9)(v10)
\Edge[lw=0.03cm](v9)(v13)
\Edge[lw=0.03cm](v10)(v19)
\Edge[lw=0.03cm](v10)(v11)
\Edge[lw=0.03cm](v11)(v18)
\Edge[lw=0.03cm](v12)(v23)
\Edge[lw=0.03cm](v12)(v15)
\Edge[lw=0.03cm](v13)(v22)
\Edge[lw=0.03cm](v13)(v14)
\Edge[lw=0.03cm](v14)(v25)
\Edge[lw=0.03cm](v15)(v24)
\Edge[lw=0.03cm](v16)(v19)
\Edge[lw=0.03cm](v16)(v21)
\Edge[lw=0.03cm](v17)(v18)
\Edge[lw=0.03cm](v17)(v20)
\Edge[lw=0.03cm](v18)(v35)
\Edge[lw=0.03cm](v19)(v34)
\Edge[lw=0.03cm](v20)(v33)
\Edge[lw=0.03cm](v20)(v23)
\Edge[lw=0.03cm](v21)(v32)
\Edge[lw=0.03cm](v21)(v22)
\Edge[lw=0.03cm](v22)(v30)
\Edge[lw=0.03cm](v23)(v31)
\Edge[lw=0.03cm](v24)(v35)
\Edge[lw=0.03cm](v24)(v29)
\Edge[lw=0.03cm](v25)(v34)
\Edge[lw=0.03cm](v25)(v28)
\Edge[lw=0.03cm](v26)(v27)
\Edge[lw=0.03cm](v26)(v28)
\Edge[lw=0.03cm](v26)(v31)
\Edge[lw=0.03cm](v27)(v29)
\Edge[lw=0.03cm](v27)(v30)
\Edge[lw=0.03cm](v28)(v32)
\Edge[lw=0.03cm](v29)(v33)
\end{tikzpicture}
\caption{A $3$-regular Cayley graph in $\ZZ_{36}$ whose line hypergraph is a counterexample to Lov\'asz' conjecture for $r=3$.}\label{fig:counterxample36vertices}
\end{figure}
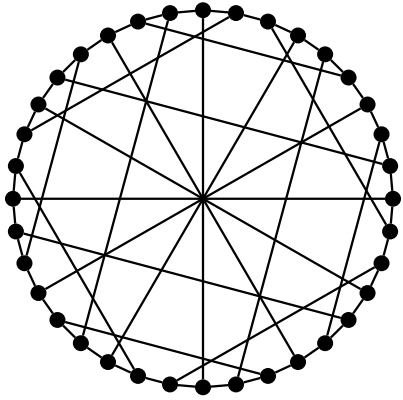

Finally, let us note that none of the examples we found contradict Ryser's conjecture or are even extremal for it, a question that has been relevant in its own right~\cite{AbuBarPokSza19, HaxNarSza18}. This is unsurprising for $r = 3$, as in this case Haxell, Narins, and Szab\'o~\cite{HaxNarSza18} showed that every extremal graph for Ryser's conjecture also fulfills Lov\'asz' conjecture. The fact that this also remains true for our counterexamples for $r = 4$ might indicate that this phenomenon could hold for any uniformity.

\medskip
\noindent\textbf{Outline.}  Section~\ref{sec:preliminaries} recalls basic notions on $r$-uniform hypergraphs, line hypergraphs, and the key statements relating the independence number of the graph to the matching number of the line hypergraph.  
Section~\ref{sec:gpproof} introduces the family of generalized Petersen graphs and proves that their line hypergraphs form infinitely many counterexamples to Lov\'asz’ conjecture for $r=3$, establishing Theorem~\ref{thm:main}.  
Section~\ref{sec:additional} introduces Cayley-generalized Petersen graphs, establishes Theorem~\ref{thm:r=4} by constructing the first quartic counterexamples, and describes the computational search used to find them.
Finally, Section~\ref{sec:conclusion} offers brief concluding remarks and open questions.

\section{Preliminaries} \label{sec:preliminaries}

Let us denote the vertex and edge set of a graph or hypergraph $G$ respectively by $V_G$ and $E_G$.
Given a graph $G$, we let $G[S]$ denote the graph induced by some  $S \subseteq V_G$, $\alpha(G)$ the \emph{independence number}, and $S(v) = \{ e \in E_G \mid v \in E \}$ the \emph{incidence set} of some vertex $v \in V_G$.
\begin{definition}
    The \emph{line hypergraph} $\mL = \mL(G)$ of a graph $G$ is defined as the hypergraph with vertex set $V_\mL = E_G$ whose hyperedges correspond to the incidence sets in $G$, that is $E_{\mL} = \{ S(v) \mid v \in V_G \}$.
\end{definition}
Note that $\mL$ is only uniform if $G$ is regular.
We can go further and also characterize the partiteness of $\mL$ through the edge-colorability of $G$.
\begin{remark}\label{rmk:partcol}
    $\mL$ is $r$-partite if and only if $G$ is $r$-regular and $r$-edge-colorable.
\end{remark}
We will also need the following trivial observation constructing a bijection between independent sets in $G$ and matchings in $\mL$.
\begin{remark}\label{rmk:indmatch}
    Every independent set $I \subseteq V_G$ in a graph $G$ defines a unique matching $M \subseteq E_\mL$ in $\mL$ of equal cardinality through $M = \{S(v)\mid v\in I\}$ and vice-versa.
\end{remark}
The main tool needed to establish our results is the following lemma, which was implicitly stated for the case $r=3$ in~\cite[Lemma 2.2]{CHM2025} and restates the property given in \autoref{conj:Lovasz} for $\mL$ in terms of independent sets in $G$.

\begin{lemma}\label{lema:indpnumbercondition}
    The hypergraph $\mL$ of an $r$-regular and $r$-edge-colorable graph $G$ is a counterexample to \autoref{conj:Lovasz} if 
    \begin{equation*}
        \alpha \Big( G - \bigcup_{e \in S} e \Big) = \alpha(G)
    \end{equation*}
    for any $S \subseteq E_G$ with $|S| = r-1$.
\end{lemma}

\begin{proof}
    Note that  $\mL$ is $r$-partite by \autoref{rmk:partcol}.
    Let $S \subseteq V_{\mL} = E_G$ with $|S| = r-1$ be arbitrary but fixed.
    The condition given for $G$ with respect to $S$ is equivalent to requiring that there exists a maximum independent set $I \subseteq V_G$ of $G$ such that $I \cap e = \emptyset$ for all $e \in S$, since any independent set in $G$ avoiding $S$ is an independent set in $G - \bigcup_{e \in S} e$ and vice-versa.
    Any such maximum independent set $I$ avoiding $\bigcup_{e \in S} e$ defines a maximum matching $M$ in $\mL$ through \autoref{rmk:indmatch} that also avoids $S$ by construction. It follows that $\nu (\mL - S) = \nu (\mL)$ and, since $S$ was arbitrary, $\mL$ therefore is a counterexample to \autoref{conj:Lovasz}.
\end{proof}

\section{Proof of \autoref{thm:main}}\label{sec:gpproof}

%
Out infinite family consists of generalized Petersen graphs, which were introduced in 1950 by Coxeter~\cite{C1950} and later given their named in 1969 by Watkins~\cite{W1969}.

\begin{definition}
    Given $n \geq 3$ and $k \in \ZZ_n$, the \emph{generalized Petersen graph} $\gp (n,k)$ is the connected cubic graph with vertex set $V_{\gp (n, k)} = \{ a_i \mid i \in \ZZ_n\} \sqcup \{b_i \mid i \in \ZZ_n\}$ and edge set
    \begin{equation*}
        E_{\gp (n, k)} = \{a_i b_i \mid i \in \ZZ_n\} \sqcup \{a_i a_{i+1} \mid i \in \ZZ_n\} \sqcup \{b_i b_{i+k} \mid i \in \ZZ_n\}.
    \end{equation*}
    For a given $i \in \ZZ_n$, we call $R_i := \{a_i, b_i\}$ the \emph{$i$-th rung} and define $S_{i,\ell} := \bigcup_{i}^{i + \ell - 1} R_j$ as the union of $\ell \in \ZZ_n$ consecutive rungs starting at $i$.
\end{definition}


In order to prove \autoref{thm:main}, it suffices to construct an infinite family of cubic graphs satisfying the hypotheses from \autoref{lema:indpnumbercondition}. We will show that the generalized Petersen graphs $\gp (5k + 11, 2)$ form such an infinite family for $k \geq 0$. We will use the following known structural results. 

\begin{theorem}\cite[Theorem 2.5]{FGS2011}
$\alpha\big( \gp (n,2) \big)=\lfloor 4n / 5\rfloor$ for any $n\geq 5$.
\end{theorem}
Note that this in particular means $\alpha\big(\gp (5k+11,2) \big) = 4k + 8$ for any $k \ge 0$.

\begin{lemma}\cite[Lemma 2.4]{FGS2011}\label{lemma:petersen_segment_of_length_5_has_independence_number_4}
    $\alpha \big( \gp (n, 2)[S_{i, 5}] \big) \leq 4$ for any  $n\geq 5$ and $ i \in \ZZ_n$.
\end{lemma}

\begin{remark}\label{rmk:edgecol}
    Castagna and Prins~\cite{Castagna_Prins_1972} showed that all generalized Petersen graphs except $\gp(5,2)$ have a Tait coloring, that is a $3$-edge-coloring where each color is incident to each vertex. In particular, this implies that they are $3$-edge-colorable.
\end{remark}

We will also need the following lemma, the intuition behind which is simple: every edge of $\gp (n,2)$ is contained in a segment of at most $3$ consecutive rungs, so after removing $2$ edges we are left with at least $n-6$ complete rungs, which must contain at least one large enough consecutive segment.

\begin{lemma}\label{lemma:petersen_find_free_consec_rungs}
    Given any two edges $uv, wx \in E_{\gp (n, 2)}$, there exists some $i_0 \in \ZZ_n$ such that $\{u, v, w, x\} \cap S_{i_0, 10} = \emptyset$ as long as $n \geq 26$.
\end{lemma}

\begin{proof}
    Given a vertex $z \in \gp (n,2)$, let $r(z)$ denote the \emph{rung of $z$}, that is, the $i \in \ZZ_n$ such that $z \in R_i$.
    Note that we will imply a natural order $1 < 2 < \ldots < n$ on the rungs.
    We may w.l.o.g. assume that (i) $\min\{r(u), r(v), r(w), r(x) \} = r(u) = 1$, (ii) $r(v) \in \{1, 2, 3\}$, and (iii) $r(w) \leq r(x)$, that is $r(x) \in \{r(w), r(w)+1, r(w)+2\}$.
    Let
    \begin{equation*}
         I = \left\{\begin{array}{lr}
         \{4, \ldots, \lfloor n/2 \rfloor \}& \text{if } r(w) > \lfloor n/2  \rfloor,\\
        \{ \lfloor n/2 \rfloor + 3, \ldots, n\} & \text{if } r(w) \le \lfloor n/2 \rfloor.
        \end{array}\right.
    \end{equation*}
    In both cases, we have $r(u) \notin I$ by (i), $r(v) \notin I$ by (ii), $r(w) \notin I$ by definition of $I$, and $r(x) \notin I$ by (iii).
    Since $|I| \geq 10$ in either case as long $n \geq 26$, we can simply take $i_0 = 4$ in the first and $i_0 = \lfloor n/2\rfloor + 3$ in the second case.
\end{proof}

The main step is encapsulated in the following theorem, which verifies the independence set condition of \autoref{lema:indpnumbercondition}.

\begin{theorem}\label{thm:petersen_are_counterexamples}
    Given any two edges $uv, wx \in E_{\gp (5k + 11, 2)}$ for $k \ge 0$, we have
    \begin{equation*}
        \alpha(\gp (5k + 11, 2) -\{u, v, w, x\}) = \alpha(\gp (5k + 11, 2)).
    \end{equation*}
\end{theorem}

\begin{proof} 
    The cases $k = 0, 1, 2$ are easily verified computationally.\footnote{The necessary code for this can be found online at \href{https://github.com/FordUniver/agps-lovasz-2025/}{github.com/FordUniver/agps-lovasz-2025}.} We will proceed inductively, that is we assume that the statement holds for some $k \geq 2$ and show that it then also holds for $k+1$. Let $n = 5(k+1) + 11$ and $uv, wx \in E_{\gp (n, 2)}$ be arbitrary but fixed.  Since $k \ge 2$ we have $n \geq 26$ and can apply \autoref{lemma:petersen_find_free_consec_rungs}, that is we can find an $i_0 \in \ZZ_n$ such that $\{u, v, w, x\} \cap S_{i_0, 10} = \emptyset$. Let us w.l.o.g. assume that $i_0 = n - 9$.

    Consider the graph $\tilde G$ with vertex set $\bigcup_{i = 1}^{n-5} R_i$ and edges induced by $\gp (n, 2)$ together with $a_{n-5}a_1$, $b_{n-5}b_2$, and $b_{n-6}b_1$. This means we removed the last $5$ rungs of $G$ and {\lq}glued{\rq} it back together to again form a generalized Petersen graph, that is $\tilde G \cong \gp (5k+11, 2)$. We can apply the inductive hypothesis and get an independent set $\tilde I \subseteq V(\tilde G)$ satisfying $\{u, v, w, x\} \cap \tilde I = \emptyset$ and $\vert \tilde I \vert = 4k + 8$.
    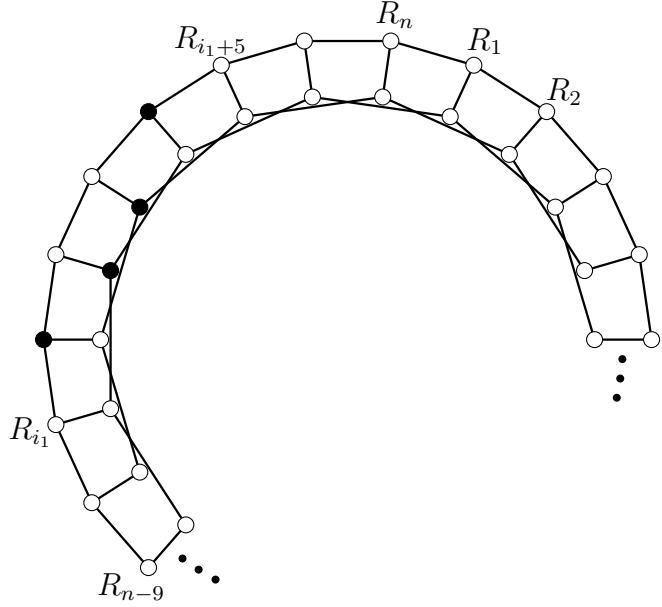
\begin{figure}[htp]
    \centering
    \begin{tikzpicture}[scale=0.5]
      \def\N{22}        
      \def\step{2}      
      \def\Rout{8}      
      \def\Rin{6.5}     
    
      \tikzset{
        VertexStyle/.style = {shape=circle, draw,
                              minimum size = 6pt, inner sep = 0pt}
      }
    
      \foreach \i in {0,...,14} {%
        \pgfmathsetmacro\theta{360/\N*\i}
    
        \pgfmathsetmacro\xO{\Rout*cos(\theta)}
        \pgfmathsetmacro\yO{\Rout*sin(\theta)}
        \Vertex[NoLabel,x=\xO cm,y=\yO cm]{u\i}
    
        \pgfmathsetmacro\xI{\Rin*cos(\theta)}
        \pgfmathsetmacro\yI{\Rin*sin(\theta)}
        \Vertex[NoLabel,x=\xI cm,y=\yI cm]{v\i}
      }
    
        \pgfmathsetmacro\theta{360/\N*14}
        \pgfmathsetmacro\xO{(\Rout+0.7)*cos(\theta)}
        \pgfmathsetmacro\yO{(\Rout+0.7)*sin(\theta)}
        \node at (\xO,\yO)   (a) {$R_{n-9}$};
        
        \pgfmathsetmacro\theta{360/\N*12}
        \pgfmathsetmacro\xO{(\Rout+0.7)*cos(\theta)}
        \pgfmathsetmacro\yO{(\Rout+0.7)*sin(\theta)}
        \node at (\xO,\yO)   (a) {$R_{i_1}$};
    
        \pgfmathsetmacro\theta{360/\N*7}
        \pgfmathsetmacro\xO{(\Rout+0.7)*cos(\theta)}
        \pgfmathsetmacro\yO{(\Rout+0.7)*sin(\theta)}
        \node at (\xO,\yO)   (a) {$R_{i_1+5}$};
    
        \pgfmathsetmacro\theta{360/\N*5}
        \pgfmathsetmacro\xO{(\Rout+0.7)*cos(\theta)}
        \pgfmathsetmacro\yO{(\Rout+0.7)*sin(\theta)}
        \node at (\xO,\yO)   (a) {$R_{n}$};
    
        \pgfmathsetmacro\theta{360/\N*4}
        \pgfmathsetmacro\xO{(\Rout+0.7)*cos(\theta)}
        \pgfmathsetmacro\yO{(\Rout+0.7)*sin(\theta)}
        \node at (\xO,\yO)   (a) {$R_{1}$};
    
        \pgfmathsetmacro\theta{360/\N*3}
        \pgfmathsetmacro\xO{(\Rout+0.7)*cos(\theta)}
        \pgfmathsetmacro\yO{(\Rout+0.7)*sin(\theta)}
        \node at (\xO,\yO)   (a) {$R_{2}$};
    
    
    \tikzset{VertexStyle/.style = {shape = circle,fill = black,minimum size = 6pt, inner sep = 0pt}}
    
        \pgfmathsetmacro\theta{360/\N*11}
        \pgfmathsetmacro\xO{\Rout*cos(\theta)}
        \pgfmathsetmacro\yO{\Rout*sin(\theta)}
        \Vertex[x=\xO cm,y=\yO cm]{}
    
        \pgfmathsetmacro\theta{360/\N*10}
        \pgfmathsetmacro\xO{\Rin*cos(\theta)}
        \pgfmathsetmacro\yO{\Rin*sin(\theta)}
        \Vertex[x=\xO cm,y=\yO cm]{}
    
        \pgfmathsetmacro\theta{360/\N*9}
        \pgfmathsetmacro\xO{\Rin*cos(\theta)}
        \pgfmathsetmacro\yO{\Rin*sin(\theta)}
        \Vertex[x=\xO cm,y=\yO cm]{}
    
        \pgfmathsetmacro\theta{360/\N*8}
        \pgfmathsetmacro\xO{\Rout*cos(\theta)}
        \pgfmathsetmacro\yO{\Rout*sin(\theta)}
        \Vertex[x=\xO cm,y=\yO cm]{}
        
      \foreach \i in {0,...,12} {%
        \pgfmathtruncatemacro\next{mod(\i+1,\N)}   
        \pgfmathtruncatemacro\jump{mod(\i+\step,\N)}
    
        \Edge[lw=0.03cm](u\i)(u\next)
    
        \Edge[lw=0.03cm](u\i)(v\i)
    
        \Edge[lw=0.03cm](v\i)(v\jump)
      }
    
    \Edge[lw=0.03cm](u13)(u14)
    \Edge[lw=0.03cm](u13)(v13)
    \Edge[lw=0.03cm](u14)(v14)

    \tikzset{VertexStyle/.style = {shape = circle,fill = black,minimum size = 3pt, inner sep = 0pt}}

    \pgfmathsetmacro\theta{360/(\N)*(\N-1+0.25)}
        \pgfmathsetmacro\xO{(\Rout+\Rin)*cos(\theta)/2}
        \pgfmathsetmacro\yO{(\Rout+\Rin)*sin(\theta)/2}
        \Vertex[x=\xO cm,y=\yO cm]{}
    \pgfmathsetmacro\theta{360/(\N)*(\N-1+0.5)}
        \pgfmathsetmacro\xO{(\Rout+\Rin)*cos(\theta)/2}
        \pgfmathsetmacro\yO{(\Rout+\Rin)*sin(\theta)/2}
        \Vertex[x=\xO cm,y=\yO cm]{}
    \pgfmathsetmacro\theta{360/(\N)*(\N-1+0.75)}
        \pgfmathsetmacro\xO{(\Rout+\Rin)*cos(\theta)/2}
        \pgfmathsetmacro\yO{(\Rout+\Rin)*sin(\theta)/2}
        \Vertex[x=\xO cm,y=\yO cm]{}
    
        \pgfmathsetmacro\theta{360/(\N)*(14+0.25)}
        \pgfmathsetmacro\xO{(\Rout+\Rin)*cos(\theta)/2}
        \pgfmathsetmacro\yO{(\Rout+\Rin)*sin(\theta)/2}
        \Vertex[x=\xO cm,y=\yO cm]{}
    \pgfmathsetmacro\theta{360/(\N)*(14+0.5)}
        \pgfmathsetmacro\xO{(\Rout+\Rin)*cos(\theta)/2}
        \pgfmathsetmacro\yO{(\Rout+\Rin)*sin(\theta)/2}
        \Vertex[x=\xO cm,y=\yO cm]{}
    \pgfmathsetmacro\theta{360/(\N)*(14+0.75)}
        \pgfmathsetmacro\xO{(\Rout+\Rin)*cos(\theta)/2}
        \pgfmathsetmacro\yO{(\Rout+\Rin)*sin(\theta)/2}
        \Vertex[x=\xO cm,y=\yO cm]{}
      
    \end{tikzpicture}
    \caption{Illustration for the proof of \autoref{thm:petersen_are_counterexamples}. Note that by construction of the independent set $I$ no vertex from $R_{i_1}$ or $R_{i_1+5}$ is selected for $I$. Hence the four marked vertices defined by \ref{edgesSi0ff} are not adjacent to any vertex in $I$.}
    \label{fig:proofillustr}
    \end{figure}
    By \autoref{lemma:petersen_segment_of_length_5_has_independence_number_4}, $\tilde I$ can only intersect at most 4 of the sets $R_{n-9}, \dots, R_{n-5}$. Let $i_1 \in [n-9, n-5]$ be such that $R_{i_1} \cap \tilde I = \emptyset$. We will construct an independent set $I \subseteq V(G)$ of size $\vert I \vert = 4k + 12$ in the following way:
    \begin{enumerate}[label = (\roman*)]\itemsep0pt
        \item For all $z \in \tilde I$ such that $z \in S_{1, i_1}$, add $z$ to $I$.\label{edgesS1i0}
        \item For all $a_i \in \tilde I$ such that $i > i_1$, add $a_{i+5}$ to $I$.\label{edgesenda}
        \item For all $b_i \in \tilde I$ such that $i > i_1$, add $b_{i+5}$ to $I$.\label{edgesendb}
        \item Add $a_{i_1 + 1}$, $b_{i_1 + 2}$, $b_{i_1 + 3}$ and $a_{i_1 + 4}$ to $I$.\label{edgesSi0ff}
    \end{enumerate}
    Since $\{u,v,w,x\}\subseteq S_{1,i_1}$, $\tilde{I}\cap\{u,v,w,x\}=\emptyset$ and $I\cap S_{1,i_1}=\tilde{I}\cap S_{1,i_1}$ we have $I \cap \{u, v, w, x\} = \emptyset$. Furthermore, $I$ is an independent set in $G$. Indeed, $I\cap S_{1,i_1}$ is an independent set in $G[S_{1,i_1}]$, since it only includes vertices chosen in \ref{edgesS1i0}. $I\cap S_{i_1,5}$ is an independent set in $G[S_{i_1,5}]$, since it only includes vertices chosen in \ref{edgesSi0ff}. $I\cap S_{i_1+5,(n-i_1-5+1)}$ is an independent set in $G[S_{i_1+5,(n-i_1-5+1)}]$, since it only includes vertices chosen in \ref{edgesenda} and \ref{edgesendb}. Therefore the only crucial edges left to check are the nine edges connecting the three segments. Since $R_{i_1}\cap I=\emptyset$ and $b_{i_1+1}\notin I$ none of the three edges $a_{i_1}a_{i_1+1},b_{i_1}b_{i_1+2}$, and $b_{i_1-1}b_{i_1+1}$ are contained in $I$. Similarly, since $R_{i_1+5}\cap I=\emptyset$ and $b_{i_1+4}\notin I$ none of the three edges $a_{i_1+5}a_{i_1+6},b_{i_1+5}b_{i_1+7}$, and $b_{i_1+4}b_{i_1+6}$ are contained in $I$, see also Figure~\ref{fig:proofillustr}. Lastly, by \ref{edgesenda} and \ref{edgesendb} none of the edges $a_na_1,b_{n-1}b_1$, and $b_{n}b_2$ is contained in $I$, as otherwise one of the edges $a_{n-5}a_1$, $b_{n-6}b_1$, and $b_{n-5}b_2$ would be contained in $\tilde{I}$ contradicting that $\tilde{I}$ is independent in $\tilde{G}$.
    Thus $\alpha(\gp (n,2) - \{u, v, w, x\}) \geq \vert I \vert = 4k + 12 = \alpha(\gp (n,2))$. Since deleting vertices can only decrease the independence number, we conclude that $\alpha(\gp (n,2) - \{u, v, w, x\}) = \alpha(\gp (n,2))$.
\end{proof}

\begin{proof}[Proof of \autoref{thm:main}]
    $\gp (5k+11, 2)$ is $3$-edge-colorable by \autoref{rmk:edgecol} and the condition of \autoref{lema:indpnumbercondition} is satisfied by \autoref{thm:petersen_are_counterexamples}, making its line hypergraph a valid counterexample to \autoref{conj:Lovasz} for any $k \geq 0$.
\end{proof}

\begin{table}[ht]
  \centering
  \caption{Overview of counterexamples for $r = 3$. $\bG^{(36)}_{3}$ 
  has generators $f_1$ of order $4$, $f_2$ of order $3$, $f_3$ of order $2$, and $f_4$ of order $3$.}
  \label{tab:r3examples}
  \begin{tabular}{@{} {l} T{8.2cm} @{}}
    \toprule
    \textbf{description} & \textbf{properties} \\ \midrule
    Biggs-Smith~\cite{CHM2025} & order $102$, diameter $7$, girth $9$, edge- and distance-transitive, Hamiltonian\smallskip \\ 
    {\tt F168D}~\cite{CHM2025} & order $168$, diameter $9$, girth $7$, edge-transitive, Hamiltonian\smallskip \\ 
    see \autoref{fig:counterxample36vertices} & order $36$, diameter $5$, girth $7$, Cayley graph in $\bG^{(36)}_{39}$, Hamiltonian\smallskip \\ 
    $\gp (19, 7)$ & order $38$, diameter $5$, girth $7$, Hamiltonian \smallskip\\
    $\gp (20, 6)$ & order $40$, diameter $6$, girth $7$, Hamiltonian \smallskip\\ 
    $\gp (25, 9)$ & order $50$, diameter $6$, girth $7$, Hamiltonian \smallskip \\ 
    $\gp (26, 8)$ & order $52$, diameter $7$, girth $7$, Hamiltonian \smallskip \\ 
    $\gp (31, 11)$ & order $62$, diameter $7$, girth $7$, Hamiltonian \smallskip \\ 
    $\gp (32, 10)$ & order $64$, diameter $8$, girth $7$, Hamiltonian \smallskip \\ 
    $\cgp (2, \bG^{(36)}_{3}; f_2 \, f_3 \, f_4, f_1 \, f_2; f_2^2, f_2)$ & order $72$, diameter $6$, girth $6$, Hamiltonian \smallskip \\ 
    $\gp (37, 17)$ & order $74$, diameter $8$, girth $7$, Hamiltonian \smallskip \\ 
    $\gp (38, 12)$ & order $76$, diameter $9$, girth $7$, Hamiltonian \smallskip \\ 
    $\gp (43, 15)$ & order $86$, diameter $9$, girth $7$, Hamiltonian \smallskip \\ 
    $\gp (44, 14)$ & order $88$, diameter $10$, girth $7$, Hamiltonian \smallskip \\ 
    $\gp (5k + 11, 2)$ & order $10(k+2) + 2$, Hamiltonian iff $n \not\equiv 5 \mod 6$~\cite[Theorem 1]{alspach1983classification} \smallskip \\ 
    \bottomrule
  \end{tabular}
\end{table}

\section{Additional counterexamples}\label{sec:additional}

Let us now describe the additional counterexamples beyond those given in the previous section.
The properties of all of the graphs have been verified computationally.
The code needed to do so along with {\tt graph6} descriptions of the graphs in this section can be found online at \href{https://github.com/FordUniver/agps-lovasz-2025/}{github.com/FordUniver/agps-lovasz-2025}.
To describe them, let us define the following family of $3$- and $4$-regular graphs.

\begin{definition} \label{def:hypergp}
     Given $m \ge 2$, some finite multiplicatively-written group $\bG$, elements $k_i\in \bG \setminus \{1\}$ with $k_i \ne (k_i)^{-1}$, and $c_i \in \bG$ for $i \in \ZZ_m$, the \emph{Cayley-generalized Petersen graph} $\cgp (m, \bG; k_i; c_i)$ has vertex set $\ZZ_m \times \bG$ and connects vertex $(i, j)$ both to $(i, j \,  k_i)$ and $(i + 1, j \, c_i)$ for any $i \in \ZZ_m$ and $j \in \bG$.
\end{definition}

Using this definition, we can list our results for $r = 3$ in \autoref{tab:r3examples} and those for $r = 4$ in \autoref{tab:r4examples}.  Note that we denote the $i$-th small group of order $n$ by $\bG^{(n)}_i$ and write it multiplicatively.

\begin{table}[ht]
  \centering
  \caption{Overview of counterexamples for $r = 4$. The groups $\ZZ_{28}$, $\ZZ_{35}$ and $\ZZ_{51}$ are written additively. The group $\bG^{(42)}_{3}$ is $\operatorname{C}_7 \times S_3$ and has generators $f_1$ of order $2$, $f_2$ of order $7$, and $f_3$ of order $3$. The group $\bG^{(42)}_{4}$ is $\operatorname{C}_3 \times \operatorname{D}_{14}$ and has generators $f_1$ of order $2$, $f_2$ of order $3$, and $f_3$ of order $7$. The group $\bG^{(48)}_{29}$ is $\operatorname{GL}(2,3)$ and has generators $f_1$ of order $2$, $f_2$ of order $3$, $f_3$ of order $4$, $f_4$ of order $4$, and $f_5$ of order $2$. The group $\bG^{(48)}_{32}$ is $\operatorname{C}_2 \times \operatorname{GL}(2,3)$ and has generators $f_1$ of order $2$, $f_2$ of order $3$, $f_3$ of order $4$, $f_4$ of order $4$, and $f_5$ of order 2.}
  \label{tab:r4examples}
  \begin{tabular}{@{}  {l} T{8cm} @{}}
    \toprule
    \textbf{description} & \textbf{properties} \\ \midrule
    $\cgp (2,\ZZ_{35}; 5,7; 15, 0)$, see \autoref{fig:counterxample70vertices} & order $70$, diameter $5$, girth $5$, Hamiltonian\smallskip \\ 
    $\cgp (2, \bG^{(42)}_{3}; f_2 \, f_3, f_2^2 \, f_3; f_1 \, f_2, 1)$ & order $84$, diameter $5$, girth $5$, Hamiltonian \smallskip\\
    $\cgp (2, \bG^{(42)}_{4}; f_2 \, f_3, f_2 \, f_3^2; f_1 \, f_2, 1)$ & order $84$, diameter $5$, girth $5$, Hamiltonian \smallskip\\ 
    $\cgp (3, \ZZ_{28}; 18, 19, 5; 18, 0, 14)$ & order $84$, diameter $6$, girth $5$, Hamiltonian \smallskip \\ 
    $\cgp (2, \bG^{(48)}_{29}; f_2 \, f_5, f_2 \, f_4; f_1 \,f_4, 1)$ & order $96$, diameter $5$, girth $6$, Cayley graph in $\bG^{(96)}_{193}$, Hamiltonian \smallskip \\
    $\cgp(2, \bG^{(48)}_{32}; f_1 f_2, f_1 f_2 f_4; f_1 f_2 f_2^2 \, f_3)$ & order $96$, diameter $5$, girth $6$, Cayley graph in $\bG^{(96)}_{200}$, Hamiltonian \smallskip \\  
    $\cgp (2, \ZZ_{51}; 3, 26 ; 47, 0 )$ & order $102$, diameter $6$, girth $5$, Hamiltonian \smallskip \\
    \bottomrule
  \end{tabular}
\end{table}

Let us also give some additional motivation for \autoref{def:hypergp} as well as a description of how we computationally explored constructions based on it.
Clow, Haxell, and Mohar~\cite{CHM2025} used {\tt nauty}~\cite{mckay1981practical, mckay2014practical} to verify that no cubic graph on at most $20$ vertices yields a counterexample to Lov\'asz' conjecture. They additional checked all cubic edge-transitive graphs of order up to $300$~\cite{conder2002trivalent, conder2006census, conder2025edge} and determined that the Bigs-Smith graph and {\tt F168D} are the only two such graphs satisfying the requirements of \autoref{lema:indpnumbercondition}.

We initially verified whether the property of being edge-transitive was relevant and if counterexamples based on Cayley graphs might exist by checking the library of vertex-transitive graphs of order at most $47$~\cite{holt2020census}, which yielded the cubic $36$-vertex example given in \autoref{fig:counterxample36vertices}.
%
%
The library contains no quartic example. Vertex-transitive graphs of order $48$ have been determined~\cite{holt2022transitive} but do not appear easily available.
We then checked cubic graphs of order $22$ using {\tt nauty}, which yielded the cubic $22$-vertex example given in \autoref{fig:counterxample22vertices}. 
%
%
The House of Graphs database~\cite{Coolsaet2023HoG2} identified it as $\gp (11, 2)$ and searching through generalized Petersen graphs suggested the family described in \autoref{sec:gpproof}, but did also yield further examples, see \autoref{tab:r3examples}.

Based on this, a natural place to look for further counterexamples to $r = 3$ as well as counterexamples for higher uniformities are therefore alternative and even wider reaching generalizations of Petersen graphs. The most notable of these in the literature are permutation graphs and supergeneralized Petersen graphs.

\begin{definition}[Permutation graphs~\cite{chartrand1967planar, ringeisen1984cycle}]
    For a given graph $G$ and a permutation $\alpha$ of its vertices, the \emph{permutation graph} $\p(G, \alpha)$ consists of two disjoint copies $G_1$ and $G_2$ of $G$ where vertex $v \in G_1$ is connected to $\alpha(v) \in G_2$ by an edge. If $G$ is a cycle on $n$ vertices, then the resulting graph is called a \emph{cycle permutation graph} and denoted by $\cyp (n, \alpha)$.
\end{definition}

\begin{definition}[Supergeneralized Petersen graphs~\cite{saravzin2007generalizing}]
    Given integers $m \ge 2$, $n \ge 3$, and $k_i\in \ZZ_n \setminus \{0\}$ with $k_i \ne -k_i$ for $i \in \ZZ_m$, the \emph{supergeneralized Petersen graph} $\sgp (m, n; k_i)$ has vertex set $\ZZ_m \times \ZZ_n$ and connects the vertex $(i, j )$ both to $(i, j + k_i )$ and to $(i + 1, j )$ for any $i \in \ZZ_m$ and $j \in \ZZ_n$.
\end{definition}

\begin{remark}
    Note that for $d$-regular $G$ and any permutation, the permutation graph is $(d+1)$-regular. In particular, cycle permutation graphs are $3$-regular. Supergeneralized Petersen on the other hand are either $3$-regular if $m = 2$ and $4$-regular for $m \ge 3$. Clearly $\gp (n, k) = \sgp (2, n; 1, k)$ for any $n$ and $k$ and $\gp (n, k) = \cyp (n, \alpha)$ with $\alpha(x) = kx$ if $k$ and $n$ are co-prime.
\end{remark}

Permutation graphs are likely to be a very good source of counterexamples to Lovasz' conjecture, but proved computationally too expensive to explore without any additional assumptions. Supergeneralized Petersen graphs can be easily explored computationally, but proved too restrictive and yielded no counterexamples not already covered by generalized Petersen graphs. \autoref{def:hypergp} constitutes a middle ground between the two, and we explored it by implementing an exhaustive generation of these constructions for given $m$ and $\bG$ in Python, relying on {\tt SageMath}~\cite{SageMath2025} and in particular the small groups library~\cite{SmallGrp} in {\tt GAP}~\cite{GAP2024} as well as {\tt cliquer}~\cite{Cliquer} to check the conditions of \autoref{lema:indpnumbercondition} and {\tt bliss}~\cite{JunttilaKaski:ALENEX2007, JunttilaKaski:TAPAS2011} for graph canonization. We mostly explored $m = 2$ with groups of order up to around $50$ as well as $m = 3$ with groups of order up to around $30$.
A brief exploratory search for $m = 4$ and $m = 5$ did not yield anything interesting. 

\section{Concluding remarks}\label{sec:conclusion}


Clow, Haxell, and Mohar~\cite[Conjecture 4.2]{CHM2025} already asked if a weaker version of the conjecture holds that still would imply Ryser's conjecture, i.e, if one can always find $k \, (r-1)$ vertices for some $1 \le k \le r-1$ whose removal decreases the matching number by at least $k$. It is unclear at this point if any of the counterexamples described here contradict this weakened version of Lov\'asz' conjecture. We verified computationally that it still holds for some $r = 3$ counterexamples of smaller order. Going to graphs of higher order or considering the counterexamples for $r = 4$ may yield different results, but doing so computationally does not seem feasible and our understanding of these constructions at this point is insufficient to decide if the conjecture holds for them purely theoretically.

\subsection*{Acknowledgements}

We thank Patrick Morris for bringing this problem to our attention. This research has been partially supported by the Spanish Ministry of Science, Innovation and Universities grant PID2023-147202NB-I00.
\noindent Aida Abiad is supported by NWO (Dutch Research Council) through the grants VI.Vidi.213.085 and OCENW.KLEIN.475.

\noindent The research of Frederik Garbe has been partially supported by the Deutsche \linebreak Forschungsgemeinschaft (DFG, German Research Foundation) through project no.\ 428212407. 
\noindent Xavier Povill is supported by the Grant PID2023-147202NB-I00 funded by MICIU/AEI/10.13039/501100011033, as well as a FPI-UPC grant from Universitat Politècnica de Catalunya and
Banco Santander.

\printbibliography

\end{document}